\documentclass[10pt]{amsart}

\usepackage{epsf,latexsym,amsmath,amssymb,amscd,datetime}
\usepackage{amsmath,amsthm,amssymb,enumerate,eucal,url}

\usepackage{graphicx}
\usepackage{color}
\newenvironment{jfnote}{ \bgroup \color{blue} }{\egroup}
\newcommand{\rhonew}{\rho^{\mathrm{new}}}
\newcommand{\specnew}{\Spec^{\mathrm{new}}}
\newcommand{\Specnew}{\Spec^{\mathrm{new}}}

\usepackage{mathrsfs}
\usepackage{amssymb}
\usepackage{dsfont}
\usepackage{verbatim}
\usepackage{url}

\newcommand{\Edir}{E^{\mathrm{dir}}}

\theoremstyle{plain}
\newtheorem{theorem}{Theorem}[section]

\newtheorem{proposition}[theorem]{Proposition}

\newtheorem{conjecture}[theorem]{Conjecture}

\theoremstyle{definition}
\newtheorem{definition}[theorem]{Definition}

\newtheorem{notation}[theorem]{Notation}

\newtheorem{remark}[theorem]{Remark}

\newcommand{\ignore}[1]{}

\newcommand{\complex}{{\mathbb C}}

\newcommand{\prob}[2]{{\PP}_{#1}{\left[\; #2\; \right]}}

\newcommand\EE{\mathbb{E}}

\newcommand\II{\mathbb{I}}

\newcommand\PP{\mathbb{P}}

\DeclareMathAlphabet{\mathcal}{OMS}{cmsy}{m}{n}

\newcommand\cC{\mathcal{C}}

\newcommand\cG{\mathcal{G}}
\newcommand\cH{\mathcal{H}}
\newcommand\cI{\mathcal{I}}

\newcommand\cL{\mathcal{L}}

\newcommand\cP{\mathcal{P}}

\newcommand\cS{\mathcal{S}}

\newcommand\frakp{\mathfrak{p}}

\newcommand{\expect}[2]{{\EE}_{#1} \left[ {#2} \right] }

\DeclareMathOperator{\Line}{Line}
\def\from{\colon}

\def\eqdef{\overset{\text{def}}{=}}

\DeclareMathOperator{\id}{id}

\DeclareMathOperator{\Spec}{Spec}

\DeclareMathOperator{\Tr}{Tr}
\DeclareMathOperator{\trace}{Tr}

\DeclareRobustCommand
  \rddots{\mathinner{\mkern1mu\raise\p@
    \vbox{\kern7\p@\hbox{.}}\mkern2mu
    \raise4\p@\hbox{.}\mkern2mu\raise7\p@\hbox{.}\mkern1mu}}

\newcommand\xhookrightarrow[2][]{\ext@arrow 0062{\hookrightarrowfill@}{#1}{#2}}
\def\hookrightarrowfill@{\arrowfill@\lhook\relbar\rightarrow}

\definecolor{dblue}{rgb}{0.0,0.0,0.5}
\definecolor{dred}{rgb}{0.5,0.0,0.0}
\usepackage[pdftex,colorlinks,citecolor=dred,urlcolor=dred,linkcolor=dblue]{hyperref}

\begin{document}
\title[Zeta Functions and Ramanujan Graphs]
{Formal Zeta Function Expansions and the Frequency of Ramanujan 
Graphs}

% \author{Joel Friedman \\ \today, at \currenttime (get rid of this line eventually)}
\author{Joel Friedman}
\address{Department of Computer Science, 
        University of British Columbia, Vancouver, BC\ \ V6T 1Z4, CANADA,
        and Department of Mathematics, University of British Columbia,
        Vancouver, BC\ \ V6T 1Z2, CANADA. }
\curraddr{}
\email{{\tt jf@cs.ubc.ca} or {\tt jf@math.ubc.ca}}
\thanks{Research supported in part by an NSERC grant.}
\urladdr{http://www.math.ubc.ca/~jf}

\subjclass[2010]{Primary }


% abstract should be called externally
\begin{abstract}
We show that logarithmic derivative of the Zeta function
of any regular graph
is given by a power series about infinity
whose coefficients are given in terms of the traces of powers
of the graph's Hashimoto matrix.

We then consider the expected value of this power series
over random, $d$-regular graph on $n$ vertices, with 
$d$ fixed and $n$ tending to infinity.
Under rather speculative assumptions, we make a formal calculation
that suggests that for fixed $d$ and $n$ large, this expected value
should have simple poles of residue $-1/2$ at $\pm (d-1)^{-1/2}$.
We shall explain that calculation suggests that for fixed $d$ there is
an $f(d)>1/2$ such that a $d$-regular graph on $n$ vertices is Ramanujan
with probability at least $f(d)$ for $n$ sufficiently large.

Our formal computation has a natural analogue when we consider
random covering graphs of degree $n$ over a
fixed, regular ``base graph.''  
This again suggests that for $n$ large, a
strict majority of random covering graphs are relatively Ramanujan.

We do not regard our formal calculations as providing overwhelming 
evidence regarding the frequency
of Ramanujan graphs.  However, these calculations are
quite simple, and yield intiguing suggestions
which we feel merit further
study.
\end{abstract}
\maketitle
\setcounter{tocdepth}{3}
\tableofcontents

%    Include unnumbered chapters (preface, acknowledgments, etc.) here.
% \include{}

% Content for Ihara paper

% \section*{To Do and Notes\hfill}
% \input{todo}

\section{Introduction}

In this paper we shall give a calculation that suggests that 
there should be many Ramanujan graphs of any fixed degree and any
sufficiently large number of vertices.
Our calculation is quite speculative, and makes a number of
unjustified assumptions.
However, our calculations are quite simple and give an intriguing
suggestion; we therefore find these calculations---at the very least---a
curiosity that merits
further study.

In more detail, we show that the logarithmic derivative of the 
Zeta function of a regular graph has a simple power series expansion at
infinity.  The coefficients of this power series invovle traces of
successively larger powers of the Hashimoto matrix of a graph.
For a standard model of a $d$-regular graph on $n$ vertices,
we consider the expected value of this logarithmic
derivative for fixed $d$ and large $n$.
We make a number of assumptions to make a formal computation
which suggests that for fixed $d$,
the expected number of real poles near each of $\pm (d-1)^{-1/2}$
tends to $1/2$ as $n$ tends to infinity.
Assuming this $1/2$ is caused entirely by real poles of the Zeta
function,
then the expected number of positive or negative
adjacency eigenvalues of absolute value in the interval
$[2(d-1)^{1/2},d)$ would be $1/4$ for both cases, positive and negative.
Under the likely assumption that for fixed $d$,
a random, $d$-regular graph on 
$n$ vertices will have two or more such eigenvalues with probability bounded
away from zero, then for sufficiently large $n$, a strict majority
of random graphs are Ramanujan.
We emphasize that all these conclusions are quite speculative.

Our computation involves a contour integral of the expected
logarithmic derivative near $u=(d-1)^{-1/2}$, where our formal
computation suggests that this
expected logarithmic derivative has a pole.  We write the contour
in a particular way that assumes that the residue at this pole
arises entirely from real eigenvalues of the Hashimoto matrix.
However it is conceivable that some of the complex Hashimoto eigenvalues
near $\pm 2(d-1)^{1/2}$
also contribute to this pole, in which case (if our other assumptions
are correct) the limiting expected number of positive and negative
real Zeta function poles may each be less than $1/4$.

We caution the reader that part of our contour passes lies the open ball
$|u| < (d-1)^{-1/2}$, which is one serious issue in the above
computation.  Indeed, for fixed $d$ and $n$ large, the Zeta function of
a random, $d$-regular graph on $n$ vertices has poles
throughout the circle $|u|=(d-1)^{-1/2}$; this follows easily from
the fact that a random such graph
has a bounded expected number of cycles of any fixed length, and so
its adjacency eigenvalue distribution tends to the Kesten-McKay
distribution (see \cite{mckay}).
Hence, for large $n$ the true expected Zeta function
should have a pole distribution throughout $|u|=(d-1)^{-1/2}$, which
makes it highly
speculative to work in any part of the region $|u|<(d-1)^{-1/2}$.
At the same time, this makes the residue of $-1/2$ at $\pm(d-1)^{-1/2}$
in the formal calculation all the more interesting, and is why we
may conjecture that the imaginary poles near $\pm(d-1)^{-1/2}$ may
possibly contribute this $-1/2$, if there is any true sense to this
residue.

It is interesting that the $-1/2$ comes from a computation on
the trace method for regular graphs that was essentially done
by Broder and Shamir \cite{broder}, although more justification
for our computation comes from the asymptotic expansions of
expected trace powers given later improvements of 
the Broder-Shamir method, namely
\cite{friedman_random_graphs,friedman_alon,friedman_kohler}.

As far as we know, this is the first direct application of Zeta functions
per se to graph theory.
Zeta functions of graphs arose first in the study of $\frakp$-adic
groups \cite{ihara,serre_trees}, and developed for general
graphs by Sunada, Hashimoto, and Bass (see \cite{terras_zeta},
beginning of Part 2).
Ihara's determinantal formula gave rise to what is now often called
the {\em Hashimoto matrix} of a graph, which can be used to count
{\em strictly non-backtracking} closed walks in a graph.
Although the underlying graph of the
Hashimoto matrix appeared in graph theory in the
1940's and 1960's (see \cite{deB,harary_line,knuth_oriented}),
the study of its spectral properties seems largely inspired by
the above work on graph Zeta functions.
Friedman and Kohler \cite{friedman_kohler} note that in the trace
method for random, regular graphs,
one gets better adjacency eigenvalue bounds if one first
gets 
analogous trace estimates for the Hashimoto matrix.
Furthermore, the solution to the Alon Second Eigenvalue Conjecture
\cite{friedman_alon} involved trace methods for the Hashimoto matrix
rather than for the adjacency matrix.
Hence the Hashimoto matrix---and therfore, by implication, also 
Zeta functions---have
played a vital role in graph theory.
However, we know of no previous
direct applications of Zeta functions to obtain new 
theorems or conjectures in graph theory.
Our theorems and conjectures---although they involve the Hashimoto 
matrix in their expansion at
infinity---seem to fundamentally involve Zeta functions.

We remark that the expected traces of
Hashimoto matrix powers are difficult to study directly.
Indeed, these expected traces are complicated by {\em tangles}
\cite{friedman_alon,friedman_kohler}, which are---roughly speaking---low 
probability
events that force a graph to have large, positive real Hashimoto
eigenvalues.  It is known that such tangles must be removed to
prove the Alon conjecture
\cite{friedman_alon} or its relativization \cite{friedman_kohler};
furthermore by modifying trace powers to eliminate the pathological
effect of tangles, the asymptotic expansions of expected trace powers
become much simpler.
For this reason we introduce a 
second formal power series, 
whose terms are a variant of the above terms, such that
(1) it is probably simpler to understand the terms of this
second formal power series, and
(2) we believe that the second
set of terms contain very similar information to the first.

We can generalize the above discussion to random covering 
maps of degree $n$ over a fixed, regular  ``base graph.''
Doing so gives the analogous formal computation
that suggests that for large $n$ we expect
that a majority of random cover maps to be {\em relatively
Ramanujan}.
Again, the $1/4$ we get (or $1/2$ for random, bipartite, regular
graphs) comes from the analogue of the Broder-Shamir
computation \cite{friedman_relative} for random covering maps, but is further
justified by higher order expansions
\cite{linial_puder,puder,friedman_kohler}.

The rest of this paper is organized as follows.
In Section~\ref{se:main} we desribe our main theorems and conjectures.
In Section~\ref{se:graphs} we describe our terminology
regarding graphs and Zeta functions.
In Section~\ref{se:zeta} we prove our expansion near infinity of
the logarithmic derivative of the Zeta function of a graph.
In Section~\ref{se:expected} we make a formal calculation of the
expected above logarithmic derivative and make numerous 
conjetures.
In Section~\ref{se:simpler} we describe variants of this formal
computation which we believe will be easier to study, and yet will
contain essentially the same information.
In Section~\ref{se:covering} we desribe other models of random graphs,
especially covering maps of degree $n$ over a fixed, regular base
graph.
In Section~\ref{se:numerical} we briefly describe our numerical experiments
and what previous experiments in the literature have suggested.
% In Section~\ref{se:conclusion}
% we make some concluding remarks.

\section{Main Results}
\label{se:main}

In this section we
state our main results, although we will use some terminology
to be made precise in later sections.
Let us begin with 
the notion of a random graph that we use.  For positive integers
$n$ and $d\ge 3$ we consider a random $d$-regular graph on $n$
vertices.  It is simplest to think of $d$ as an even integer with
a random graph generated by $d/2$ permutations on $\{1,\ldots,n\}$,
which we denote $\cG_{n,d}$,
as was used in
\cite{broder,friedman_random_graphs,friedman_alon}; 
our models therefore allow for multiple edges and self-loops.
We remark that
there are similar ``algebraic'' models for $d$ and $n$ of any
parity \cite{friedman_alon} which we shall
describe in Section~\ref{se:covering}.

We will give a formal calculation that indicates that for fixed $d$
the expected number of adjacency eigenvalues 
of a graph in $\cG_{n,d}$ of absolute value in $[2(d-1)^{1/2},d)$
is one-half, $1/4$ positive and $1/4$ negative;
a more conservative conjecture is that
$1/4$ is an upper bound on each side.
We also conjecture that the probability that for fixed $d$ and
$n\to\infty$, a graph of $\cG_{n,d}$
has at least two such eigenvalues is bounded from below by a positive
constant.
These two conjectures---if true---imply that for any fixed $d$, 
for all $n$ sufficiently large a strict majority of graphs in $\cG_{n,d}$
are Ramanujan.

After explaining our conjecture, we will comment on generalizations
to random covering maps of degree $n$ to a fixed, regular graph.
Then we will describe some numerical experiments we made to
test our conjecture; 
although these calculations suggest that our formal computation
may be close to the correct answer, our calculations are done
on graphs with under one million vertices (and assume that
certain software is computing correctly);
it may be 
that one needs more vertices to see the correct 
trend, and it is commonly believed
that there are fewer Ramanujan graphs than our formal
calculation
suggests; see \cite{miller_novikoff}.

Let us put our conjectures in a historical context.
For a graph, $G$, on $n$ vertices, we let
$$
\lambda_1(G) \ge \lambda_2(G) \ge \cdots \ge \lambda_n(G)
$$
be the $n$ eignevalues of $A_G$, the adjacency matrix of $G$.
In \cite{alon_eigenvalues}, 
Noga Alon conjectured that for fixed integer $d\ge 3$
and $\epsilon>0$, as $n$ tends to infinity, the probability that
a random $d$-regular graph, $G$, on $n$ has
$$
\lambda_2(G)\le 2(d-1)^{-1/2} + \epsilon
$$
tends to one.  Alon's interest in the above conjecture was that
the above condition on $\lambda_2(G)$ implies 
(\cite{gabber_galil,alon_milman,tanner})
that $G$ has a number of
interesting isoperimetric or ``expansion'' properties.
Broder and Shamir \cite{broder} introduced a trace method to study the above
question; \cite{broder,friedman_kahn_szemeredi,friedman_random_graphs}
gave high probability bounds on $\lambda_2(G)$ with
$2(d-1)^{-1/2}$ replaced with a larger constant, and
\cite{friedman_alon} finally settled the original conjecture.
We remark that all the aforementioned papers actually give stronger
bounds, namely with $\lambda_2(G)$ replaced with
$$
\rho(G) \eqdef \max_{i\ge 2} |\lambda_2(G)|.
$$

For many applications, it suffices to specify one particular graph,
$G\in\cG_{n,d}$, which satisfies the bound in Alon's conjecture.
Such graphs were given in 
\cite{lps,margulis,morgenstern}; \cite{lps} coined the term
{\em Ramanujan graph} to describe a $d$-regular graph, $G$, satisfying
$$
\lambda_i(G) \in \{d,-d\} \cup [-2(d-1)^{1/2},2(d-1)^{1/2}] 
$$
for all $i$.  This is stronger than Alon's conjecture in that
the $\epsilon$ of Alon becomes zero; it is weaker than
what trace methods prove, in that $-d$ is 
permitted as an eigenvalue, i.e., the graph may be bipartite.
Recently \cite{mss} have proven the existence of a sequence of
bipartite
Ramanujan graphs of any degree $d$ for a sequence of $n$'s tending
to infinity; \cite{lps,margulis,morgenstern} constructed sequences
for $d$ such that $d-1$ is a prime power, but the 
\cite{lps,margulis,morgenstern} are more explicit---contructible
in polynomial of $\log n$ and $d$---than those of 
\cite{mss}---constructible in polynomial of $n$ and $d$.
Our results do not suggest any obvious method of constructing
Ramanujan graphs.

Our formal calculation for $\cG_{n,d}$ is based on two results:
Ihara Zeta functions of graphs
(see \cite{terras_zeta}), and a trace estimate essentially known
since \cite{broder}.
We give a stronger conjecture based on the trace methods and
estimates of
\cite{friedman_random_graphs,friedman_alon,friedman_kohler},
which give somewhat more justification for the formal calculation
we describe.

% We believe that our calculation may represent essentially the
% first application of Zeta functions of graphs to graph theory
% that goes beyond
% the study of the Hashimoto matrix.
% For example, in the study of Alon's conjecture
% \cite{broder,friedman_kahn_szemeredi,friedman_regular_graphs,
% friedman_alon} and its relativization
% \cite{friedman_relative,linial_puder,lubetzky,a-b,puder,
% friedman_kohler},
% some of the methods essentially work with the Hashimoto matrix of
% a random graph, studying strictly non-backtracking closed walks.
% It is remarked in \cite{friedman_kohler} that when one can give
% ``equivalent'' estimates for the adjacency matrix and the Hashimoto
% matrix, then one usually gets a better bound by using trace methods
% on the Hashimoto matrix to understand the location of
% its eigenvalues, and then
% translating to adjacency eigenvalue results (rather than using the
% adjacency trace results directly).
% Hence it is known that the study of random Hashimoto matrices can
% improve results in the study of random adjacency matrices.
% Zeta functions of graphs historically gave rise to
% Hashimoto matrices
% through the 
% works of Ihara, Serre, Sunada, and others (see \cite{terras_zeta}) on
% the study of blah,
% we are unaware of any applicat
% 

Let us roughly describe our methods.  
% It is simplest to fix an even integer $d\ge 4$ and describe the
% model used by \cite{broder}, namely
% we let $\cG_{n,d}$ be the model of a random graph on $n$ vertices
% generated by $d/2$ permutations on $\{1,\ldots,n\}$, each chosen
% uniformly (from the $n!$ possibilities) and independently.
% Hence our graphs can have self-loops and multiple edges; any atom
% $G\in\cG_{n,d}$ automatically has $d$ as an eigenvalue.
We will consider the function
$$
\expect{G\in\cG{n,d}}{\zeta_G'(u)/\zeta_G(u)},
$$
i.e., the expected logarithmic derivative of $\zeta_G(u)$; 
using the results of \cite{friedman_alon,friedman_kohler} regarding
the Alon conjecture, we can
write the expected number of eigenvalues equal to or
greater than $2(d-1)^{1/2}$
as half of a contour integral of the above logarithmic
derivative near $2(d-1)^{-1/2}$ (really minus the logarithmic derivative,
since we are counting poles).  
This contour integral is 
essentially unchanged if we replace minus this logarithmic 
derivative by the simpler expression
$$
\cL_G(u) = \sum_{k=0}^\infty u^{-1-k} \Tr(H_G^k)(d-1)^{-k} ,
$$
where $H_G$ denotes the 
Hashimoto matrix of a graph, $G$.
We recall \cite{friedman_alon,friedman_kohler}
(although essentially from \cite{friedman_random_graphs}) that we can estimate
the expected traces of Hashimoto matrices as
\begin{equation}\label{eq:Hash_asymp}
\expect{G\in\cG{n,d}}{\Tr(H_G^k)}
=P_0(k)+P_1(k)n^{-1}+\cdots+ P_{r-1}(k)n^{1-r} + {\rm err}_{n,k,r}
\end{equation}
where the $P_i(k)$ are functions of $k$ and $d$, and for any fixed
$r$ we have
$$
|{\rm err}_{n,k,r}| \le C k^C n^{-r} (d-1)^k
$$
for some $C=C(r)$.
The leading term, $P_0(k)$, has been essentially known since
\cite{broder} (see
\cite{friedman_random_graphs,friedman_alon,friedman_kohler})
to be
$$
P_0(k) = (d-1)^k + (d-1)^{k/2} \II_{\rm even}(k) + O(k) \; (d-1)^{k/3},
$$
where $\II_{\rm even}(k)$ is the indicator function that $k$ is even.
If we assume that we may evaluate the expected value of
$\cL_G(u)$ by writing is as a formal sum,
term by term, using
\eqref{eq:Hash_asymp}, then the
$$
(d-1)^{k/2} \II_{\rm even}(k)
$$
term of $P_0(k)$ gives a term of the expected value of $\cL_G(u)$ equal to
$$
\frac{u}{u^2-(d-1)},
$$
whose residue at $u=\pm(d-1)^{-1/2}$ is $1/2$.  If this exchange
of summation gives the correct asymptotics, i.e., the sum of
the terms corresponding
to $P_i(k)n^{-i}$ for $i\ge 1$ tends to zero as $n\to\infty$,
then as $n\to\infty$ we would have that the expected number
of positive real poles and negative real poles is $1/2$ each (with the
$\pm(d-1)^{-1/2}$ countributing one-half times their expected
multiplicity).

It is known that for large $i$, the $P_i(k)$ are problematic due to
``tangles'' \cite{friedman_alon}, which are certain low probability events
in $\cG_{n,d}$ that force a large second adjacency eigenvalue.
Hene we might wish to modify the $P_i(k)$, as done
in \cite{friedman_alon,friedman_kohler}, by introducing a
variant of the Hashimoto trace.  
This leads us to later introduce the related functions
$\widehat P_i(k)$, which we believe will be easier to study
but contain almost the same information as the $P_i(k)$.
We shall also explain a generalization of this calculation to models
of covering graphs of degree $n$ over a fixed, regular ``base graph.''

For reasons that we explain, this formal calculation takes a number
of ``leaps of faith'' that we cannot justify at this point;
on the other hand, 
it seems like a natural formal calculation to make, and it suggests
an intriguing conjecture.

We have made a brief, preliminary experimental
investigation of this conjecture with $d=4$ and $d=6$ for positive
poles; these
experiments are not particularly conclusive:
according to \cite{miller_novikoff} the true trend may require
very large values of $n$: this is based on the 
calculations that the ``width of concentration''
of the second largest adjacency eigenvalue will eventually overtake its 
mean's distance
to $2(d-1)^{1/2}$.
However, if this width of concentration is of the same order of
magnitude as its distance to $2(d-1)^{1/2}$ as $n\to\infty$, then 
our conjecture does not contradict the other
findings of \cite{miller_novikoff}
(regarding a Tracy-Widom distribution over the width of concentration).
Our experiments are made with
graphs of only $n\le 400,000$ vertices. 
For graphs of this size or smaller
it does not
look like the expected number of real poles has
stabilized; however, in all of our experiments this expected number
is smaller than $1/4$ 
for all but very small
values of $n$.

\section{Graph Theoretic Preliminaries}
\label{se:graphs}

In this subsection we give specify our precise
definitions for a number of
concepts in graph and algebraic graph theory.  We note that such definitions
vary a bit in the literature.
For example,
in this paper graphs may have multiple edges and two types of
self-loops---half-loops and whole-loops---in the terminology of
\cite{friedman_geometric_aspects};
also see \cite{st1,st2,st3}, for example, regarding half-loops.

\subsection{Graphs and Morphisms} 

\begin{definition} A {\em directed graph}
(or
\emph{digraph}) is a tuple \( G = (V, \Edir, t, h) \) where \( V \) and
\(\Edir\) are sets---the {\em vertex} and 
{\em directed edge} sets---and \( t
\from \Edir \rightarrow V \) is the \emph{tail map} and \( h \from \Edir
\rightarrow V \) is the \emph{head map}. A directed edge \(e\) is called 
{\em self-loop}
if \( t(e) = h(e) \), that is, if its tail is its head.
Note that our definition also allows for \emph{multiple edges}, that is
directed edges with identical tails and heads.  Unless
specifically mentioned, we will only consider directed graphs which have
finitely many vertices and directed edges.
\end{definition}

A graph, roughly speaking, is a
directed graph with an involution that pairs the edges.

\begin{definition} An {\em undirected graph} (or simply a {\em graph}) is a
tuple \( G = (V, \Edir, t, h, \iota) \) where \( (V, \Edir, t, h) \) is a
directed graph and where \( \iota \from \Edir \rightarrow \Edir \), called
the {\em opposite map}
or {\em involution} of the graph,
is an involution on the set of directed edges
(that is, \( \iota^2=\id_{\Edir} \) is the identity) satisfying \( t \iota
= h \). The directed graph \( G = (V, \Edir, t, h) \) is called the
\emph{underlying directed graph} of the graph \(G\). If \(e\) is an edge,
we often write $e^{-1}$ for $\iota(e)$ and call it the \emph{opposite edge}. 
A
self-loop \(e\) is called a {\em half-loop}
if \( \iota(e) = e\), and otherwise is called a
{\em whole-loop}.

The opposite map induces an equivalence relation on the directed edges of
the graph, with $e\in \Edir$ equivalent to $\iota e$;
we call the quotient set, \( E \), the 
\emph{undirected edge} of the 
graph \(G\) (or simply its \emph{{edge}}). 
Given an edge of a
graph, an \emph{orientation}
of that edge is the choice of a representative
directed edge in the equivalence relation (given by the opposite map).
\end{definition}

\begin{notation} For a graph, $G$, we use the notation
$V_G,E_G,\Edir_G,t_G,h_G,\iota_G$ to denote
the vertex set, edge set, directed edge set, tail map, head map, 
and opposite map of $G$; similarly for directed graphs, $G$.
\end{notation}

\begin{definition}
Let $G$ be a directed graph.
The \emph{adjacency matrix}, \(A_G\), of $G$
is the 
square matrix indexed on the vertices, $V_G$, whose \( (v_1, v_2) \)
entry is the number of directed edges whose tail is the vertex \(v_1\) and
whose 
head is the vertex \(v_2\). 
The {\em indegree} (respectively {\em outdegree}) of a vertex, $v$,
of $G$ is the number
of edges whose head (respectively tail) is $v$.

The adjacency matrix of an undirected graph, $G$, is
simply the adjacency matrix of its underlying directed graph. 
For an undirected graph, the indegree of any vertex equals its outdegree,
and is just called its {\em degree}.
The {\em degree matrix} of $G$ is the diagonal matrix, 
$D_G$, indexed on $V_G$
whose $(v,v)$ entry is the degree of $v$.
We say that $G$ is {\em $d$-regular} if $D_G$ is $d$ times the identity
matrix, i.e., if each vertex of $G$ has degree $d$.
\end{definition}

For any non-negative integer $k$, the number of closed walks of length
$k$ is a graph, $G$, is just the 
trace, $\Tr(A_G^k)$, of the $k$-th power
of $A_G$.

\begin{notation}\label{no:lambda}
Given a graph, $G$, the matrix $A_G$ is symmetric, and
hence the eigenvalues of $A_G$ are real and can be ordered
$$
\lambda_1(G) \ge \cdots \ge \lambda_n(G),
$$
where $n=|V_G|$.  We reserve the notation $\lambda_i(G)$ to denote the
eigenvalues of $A_G$ ordered as above.
\end{notation}

If $G$ is $d$-regular, then $\lambda_1(G)=d$.

\begin{definition}
Let $G$ be a graph.  We define the 
\emph{{directed line graph}} or 
{\em {oriented line graph}} of $G$,
denoted 
\( \Line(G) \), to be the
directed graph \( L=\Line(G) = (V_L, \Edir_L, t_L, h_L) \) given as follows:
its vertex set, \( V_L \), is the set \( \Edir_G \) of directed edges of
\(G\); its set of directed edges is defined by \[
\Edir_L = \left\{ (e_1,e_2) \in \Edir_G \times \Edir_G \mid h_G(e_1) =
t_G(e_2) \text{ and } \iota_G(e_1) \neq e_2 \right\} \] that is,
$\Edir_L$ corresponds to the
non-backtracking walks of length two in $G$. 
The tail and head maps are simply
defined to be the projections in each component, that is by \(t_L(e_1,e_2)
= e_1 \) and \( h_L(e_1,e_2) =e_2 \).

The {\em {Hashimoto matrix}} 
of $G$ is the adjacency matrix of its 
directed line graph, denoted
$H_G$, which is, therefore, a square matrix indexed
on $\Edir_G$.
We use the symbol $\mu_1(G)$ to denote the Perron-Frobenius eigenvalue
of $H_G$, and use $\mu_2(G),\ldots,\mu_m(G)$, where $m=|\Edir_G|$,
to denote the remaining eigenvalues, in no particular order (all
concepts we discuss about the $\mu_i$ for $i\ge 2$ will not depend
on their order).
\end{definition}

If $G$ is $d$-regular, then $\mu_1(G)=d-1$.

It is easy to see that
for any positive integer $k$, the number of strictly non-backtracking
closed walks of length
$k$ in a graph, $G$, equals the trace, $\Tr(H_G^k)$, of the $k$ power
of $H_G$; of course, the strictly non-backtracking walks begin and end
in a vertex, whereas $\Tr(H_G^k)$ most naturally counts walks beginning
and ending in an edge; the correspondence between the two notions can
be seen by taking a walk of $\Line(G)$,
beginning and ending an in a directed edge, $e\in\Edir_G$, and mapping
it to the strictly non-backtracking closed walk in $G$ beginning at,
say, the tail of $e$.

For graphs, $G$, that have half-loops, the Ihara determinantal formula
takes the form (see \cite{friedman_alon,st1,st2,st3}):
\begin{equation}\label{eq:Zetahalf}
\det(\mu I -H_G) = \det\bigl(\mu^2 I - \mu A_G + (D_G-I) )
(\mu-1)^{|{\rm half}_G|}
(\mu^2-1)^{|V_G|-|{\rm pair}_G|},
\end{equation}
where ${\rm half}_G$ is the set of half-loops of $G$, and
${\rm pair}_G$ is the set of undirected edges of $G$ that are not
half-loops, i.e., the collection of sets of the form, $\{e_1,e_2\}$
with $\iota e_1=e_2$ but $e_1\ne e_2$.

\section{Variants of the Zeta Function}
\label{se:zeta}

For any graph, $G$, recall that
$A_G$ denotes its adjacency matrix, and 
$H_G$ denotes its Hashimoto matrix, i.e., the adjacency matrix of what is
commonly called $G$'s {\em oriented line graph}, $\Line(G)$.
If $\zeta_G(u)$ is the Zeta function
of $G$, then we have
$$
\zeta_G(u) = \frac{1}{\det(I-u H_G)},
$$
which, for $d$-regular $G$,
we may alternatively write via the Ihara determinantal formula
\begin{equation}\label{eq:ihara_det}
\det(I-u H_G) =
\det\Bigl(I-u A_G- u^2(d-1)\Bigr)(1-u^2)^{-\chi(G)}
\end{equation}
(provided $G$ has no half-loops, with a simple modification if
$G$ does).

\begin{definition}
Let $G$ be a $d$-regular graph.
We call an eigenvalue of $H_G$ {\em non-Ramanujan} if it is purely
real and different from
$$
\pm 1,\pm (d-1)^{1/2},\pm (d-1).
$$
We say that an eigenvalue of $A_G$ is {\em non-Ramanujan} if it
is of absolute value strictly between $2(d-1)^{1/2}$ and $d$.
It is known that the number of non-Ramanujan eigenvalues of $H_G$ is
precisely twice the number of eigenvalues of $A_d$.
$G$ is called {\em Ramanujan} if it has no non-Ramanujan $H_G$
eigenvalues,
or, equivalently, no non-Ramanujan $A_G$ eigenvalues.
Similary for positive non-Ramanujan eigenvalues of both $H_G$ and $A_G$,
and the same with ``positive'' replaced with ``negative.''
\end{definition}

\begin{notation}
For any $\epsilon,\delta>0$, let $C^+_{\epsilon,\delta}$ be the
boundary of the rectangle
\begin{equation}\label{eq:Cplus}
\{ x+iy\in\complex\ |\ |1-x(d-1)^{1/2}|\le \epsilon,\ |y|\le \delta
\};
\end{equation}
define $C^-_{\epsilon,\delta}$ similarly, with $x$
replaced with
$-x$.
\end{notation}

\begin{definition} We say that a $d$-regular graph, $G$, is 
{\em $\epsilon$-spectral} if $G$'s real Hashimoto eigenvalues lie in
set
$$
\{-(d-1),-1,1,(d-1)\} \cup \{ x \ | \  |1-x(d-1)^{-1/2}| < \epsilon\}.
$$
\end{definition}

We remark that for fixed $d\ge 3$ and $\epsilon>0$, 
it is known that a fraction $1-O(1/n)$ of
random $d$-regular graphs on $n$ vertices are
$\epsilon$-spectral \cite{friedman_alon}.

For $\epsilon$-spectral $G$ we have that
the number of real, positive Hashimoto eigenvalues is given by
$$
\frac{1}{2\pi i}
\int_{C^+_{\epsilon,\delta}} \frac{-\zeta_G'(u)}{\zeta_G(u)}\;du
$$
for $\delta>0$ sufficiently small,
where $C^+_{\epsilon,\delta}$ is traversed in the counterclockwise
direction; similarly for negative eigenvalues.

Observe that for each $G$, for large $|u|$ we have
$$
-\zeta_G'(u)/\zeta_G(u) =
\sum_{\mu\in{\rm Spec}(H_G)} -\mu\,(1-u\mu)^{-1} ,
$$
where ${\rm Spec}(H_G)$ denotes the set of eigenvalues of $H_G$, counted
with multiplicity, and hence
$$
-\zeta_G'(u)/\zeta_G(u) 
=\sum_{\mu\in{\rm Spec}(H_G)} 
\ \sum_{k=0}^\infty  u^{-1-k} \mu^{-k}.
$$
By the Ihara determinantal formula, the set of eigenvalues, $\mu$, of 
$H_G$, consists of
$\pm 1$, with multiplicity $-\chi(G)$, and, in addition, the eigenvalues
that arise as the roots, $\mu_1,\mu_2$, from an equation
$$
\mu^2 - \mu \lambda + (d-1) = 0,
$$
where $\lambda$ ranges over all
eigenvalues of $A_G$; in particular, any pair $\mu_1,\mu_2$ as such
satisfy
$\mu_1\mu_2=d-1$; hence to sum over $\mu^{-1}$
over the pairs $\mu_1,\mu_2$ as such 
is the same as summing over
$\mu/(d-1)$ of all such eigenvalues.  It easily follows that
$$
\sum_{\mu\in{\rm Spec}(H_G)} 
\mu^{-k} = \bigl( 1 + (-1)^k \bigr) n(d-2)/2  +
\Bigl( \Tr(H_G^k) - \bigl( 1 + (-1)^k \bigr) n(d-2)/2 \Bigr) (d-1)^{-k}  
$$
where $\Tr$ denotes the trace, and hence
$$
-\zeta_G'(u)/\zeta_G(u) =
\cL_G(u) + e(u),
$$
where
$$
\cL_G(u) = \sum_{k=0}^\infty u^{-1-k} \Tr(H_G^k)(d-1)^{-k},
$$
and
$$
e(u)
=
\sum_{k=0}^\infty u^{-1-k}
\bigl(1-(d-1)^{-k} \bigr) \bigl( 1 + (-1)^k \bigr) n(d-2)/2  
$$
$$
=\frac{n(d-2)}{2u} 
% \sum_{\rm even\ }k\ge 2} u^{-1-k} \bigl(1-(d-1)^{-k} \bigr) n(d-2)
\sum_{k\ge 0} \Bigl[ u^{-k} +(-u)^{-k} - \bigl( (d-1)u\bigr)^{-k} - 
\bigl( -(d-1)u\bigr)^{-k} \Bigr]
$$
$$
=\frac{n(d-2)}{2u} 
\left[ \frac{1}{1-u} + \frac{1}{1+u} - \frac{1}{1-(d-1)u}
- \frac{1}{1+(d-1)u}  \right] \  .
$$

It follows that $e(u)$ is a rational function with poles only at
$u=\pm 1$ and $\pm1/(d-1)$.  
Furthermore, the $e(u)$ poles at $\pm 1$ have residue
$-\chi(G)$.

\begin{definition}
\label{de:calL}
Let $G$ be a $d$-regular graph without half-loops.  We define the
{\em essential logarithmic derivative} to be the meromorphic function
complex function
$$
\cL_G(u) = \sum_{k=0}^\infty u^{-1-k} \Tr(H_G^k)(d-1)^{-k}.
$$
\end{definition}
We caution the reader that $\cL_G(u)$ is the interesting part of
{\em minus} the usual logarithmic derivative of $\zeta_G(u)$ (since
we are
interested in poles, not zeros).
Clearly $\Tr(H_G^k)$ is bounded by the number of non-backtracking
walks of length $k$ in $G$, i.e., $|V_G|d(d-1)^{k-1}$, and hence 
the above expansion for $\cL_G(u)$ converges for all $|u|>1$.
 
We summarize the above discussion in the follow proposition.

\begin{proposition} Let $G$ be a $d$-regular, $\epsilon$-spectral
graph.  Then, with $C^+_{\epsilon,\delta}$ as in
\eqref{eq:Cplus}, we have
that the number of positive non-Ramanujan
Hashimoto eigenvalues is given by
$$
\frac{1}{2\pi i}
\int_{C^+_{\epsilon,\delta}} \cL(u)\;du
$$
for $\delta>0$ sufficiently small; similarly
for negative eigenvalues.
\end{proposition}

We remark that if $G$ is $\epsilon$-spectral for $\epsilon>0$ small, 
then $\Tr(H_G^k)$ is 
the sum of $d^k$ plus $nd-1$ other eigenvalues, all of which
are within the ball $|\mu|\le (d-1)^{1/2}+\epsilon'$ for 
some $\epsilon'$ that tends to zero as $\epsilon$ tends to zero;
hence, for such $G$, the expression for $\cL_G$ in
Definition~\ref{de:calL} has a simple pole of residue $1$ at
$u=1$, and the power series at infinity for
$$
\cL_G(u) - \frac{1}{u-1}
$$
converges for all $|u|^{-1}<(d-1)^{1/2}+\epsilon'$.

\section{The Expected Value of $\cL_G$}
\label{se:expected}

For any even integer, $d$, and integer $n>0$, we define
$\cG_{n,d}$ to the probability space of $d$-regular random graphs
formed by independently choosing $d/2$ permutations, $\pi_1,\ldots,
\pi_{d/2}$ uniformly from the set of 
$n!$ permutations of $\{1,\ldots,\}$; to each such
$\pi_1,\ldots,\pi_{d/2}$
we associate the random graph, $G=G(\{\pi_i\})$, whose vertex set is
$V_G = \{1,\ldots,n\}$, and whose edge set, $E_G$, consists of all
sets
$$
E_G =  \{ \{i,\pi_j(i)\}\ |\ i=1,\ldots,n,\ j=1,\ldots,d/2 \}.
$$
It follows that $G$ may have multiple edges and self-loops.

The following is a corollary of 
\cite{friedman_alon,friedman_kohler}.
\begin{theorem}  
For a $d$-regular graph, $G$, define
define $N_A^+(G)$ to be the number of positive non-Ramanujan adjacency
eigenvalues
of $G$ plus the multiplicity
of $2(d-1)^{-1/2}$ (if any) as an eigenvalue of $G$.
Then for any even $d\ge 10$ and $\epsilon>0$ we have that
$$
\lim_{n\to\infty} \lim_{\delta\to 0}
\expect{G\in\cG_{n,d}}{ N_{A,+}(G) -
\frac{1}{4\pi i} 
\int_{C^+_{\epsilon,\delta}} \cL_G(u)\,du } = 0,
$$
where we interpret the contour integral as its Cauchy principle
for graphs, $G$, whose Zeta function has a pole on $C^+_{\epsilon,\delta}$.
\end{theorem}
\begin{proof}
For $d\ge 10$ we know that the probability that a graph
has any such eigenvalues is at most $O(1/n^2)$
(see \cite{friedman_alon} for $d\ge 12$, and
\cite{friedman_kohler} for $d=10$), and hence this expected
number is at most $1/n$.  
\end{proof}
For $d\ge 4$ one might conjecture the same theorem holds, 
although this does not seem to follow
literally from \cite{friedman_alon,friedman_kohler};
however, if one conditions on $G\in\cG_{n,d}$ 
not having a $(d,\epsilon')$-tangle
(in the sense of \cite{friedman_kohler}), which is an order
$O(1/n)$ probability event, then we get equality.
The problem is that one does not know where most of the eigenvalues
lie in graphs that have tangles; we conjecture that graphs with 
tangles will not give more than an $O(1/n)$ expected
number of eigenvalues strictly
between $2(d-1)^{1/2}+ \epsilon$ and $d$; hence we conjecture that
the above above theorem remains true for all $d\ge 4$.

We remark that for any $d$-regular graph, we have that
$$
\lim_{\delta\to 0}
\frac{1}{4\pi i} 
\int_{C^+_{\epsilon,\delta}} \cL_G(u)\,du  =
N_{A,+,\epsilon}(G),
$$
where $N_{A,+,\epsilon}$ counts the number of positive,
real Hashimoto eigenvalues, $\mu$, such that
$$
|1-\mu(d-1)^{1/2}|\le \epsilon.
$$
Hence for any $d,n,\epsilon$ we have
$$
\lim_{\delta\to 0}
\expect{G\in\cG_{n,d}}{ 
\frac{1}{4\pi i} 
\int_{C^+_{\epsilon,\delta}} \cL_G(u)\,du } =
\expect{G\in\cG_{n,d}}{N_{A,+,\epsilon}(G)}.
$$

Now we wish to conjecture a value for
\begin{equation}\label{eq:magic}
\lim_{\delta\to 0}\expect{G\in\cG_{n,d}}{
\frac{1}{4\pi i} 
\int_{C^+_{\epsilon,\delta}} \cL_G(u)\,du
} .
\end{equation}
We now seek to use trace methods to
in order to conjecture what the value of \eqref{eq:magic}
will be for fixed $d$ and
$n\to\infty$.

It is known that for $d,r$ fixed and $n$ large, we have that 
\begin{equation}\label{eq:Pdefined}
\expect{G\in\cG{n,d}}{\Tr(H_G^k)} =
P_0(k) + P_1(k) n^{-1} + \ldots + P_{r-1}(k)n^{1-r} + {\rm err}_r(n,k),
\end{equation}
where $P_i(k)$ are functions of $k$ alone, and
$$
|{\rm err}_r(n,k)| \le C_r k^{2r}(d-1)^k n^{-r}.
$$
Furthermore, $P_0(k)$ is known (see \cite{friedman_random_graphs,
friedman_alon,friedman_kohler}, but essentially since \cite{broder}) to equal
$$
P_0(k) 
= O(kd) + \sum_{k'|k} (d-1)^{k'},
$$
where $k'|k$ means that $k'$ is a positive integer dividing $k$;
in particular, we have
\begin{equation}\label{eq:Pzero}
P_0(k) = (d-1)^k + \II_{{\rm even}}(k) (d-1)^{k/2} + O(d-1)^{k/3},
\end{equation}
where $\II_{{\rm even}}(k)$ is the indicator function of $k$ being even.
Furthermore, we believe
the methods of \cite{friedman_random_graphs,friedman_kohler} 
will show that for each $i$ we have
\begin{equation}\label{eq:cPdefined}
\cP_i(u) = \sum_{k=0}^\infty u^{-1-k} P_i(k) (d-1)^{-k}
\end{equation}
is meromorphic with finitely many poles outside any disc about the
origin of radius strictly greater than $1/(d-1)$; our idea is to
use the ``mod-$S$'' function approach of \cite{friedman_kohler} 
to show that each $P_i(k)$ is a polyexponential plus an error term
(see \cite{friedman_kohler}) and
to argue for each ``type'' separately, although we have not written
and checked this carefully as of the writing of this article; hence this 
belief may be 
regarded as a (plausible) conjecture at present.
Let give some conjectures based on the above 
assumption, and the (more speculative) assumption
that we can evaluate the above asymptotic expansion
by taking expected values term by term, and the (far more speculative)
assumption that the $\cP_i(u)$, formed by summing over arbitrarily
large $k$, reflect the properties of $\cG_{n,d}$ with $n$ fixed.

\begin{definition} For any even $d\ge 4$, set
$$
N_i = 
\frac{1}{2\pi i}
\int_{C^+_{\epsilon,\delta}}  \cP_i(u)\,du ,
$$
where we assume the $\cP_i(u)$, given in \eqref{eq:cPdefined}, based on
functions $P_i(k)$ given in \eqref{eq:Pdefined}, are meromorphic functions,
at least for $|u|$ near $(d-1)^{-1/2}$.
We say that $\cG_{\;\cdot\;,d}$ is {\em positively approximable to order $r$}
if
for some $\epsilon>0$ we have
$$
\lim_{\delta\to 0}
\frac{1}{2\pi i}
\int_{C^+_{\epsilon,\delta(n)}} 
\expect{G\in\cG_{n,d}}{ \cL_G(u) }\,du 
= N_0 + N_1 n^{-1} + N_2 n^{-2} + \cdots + N_r n^{-r} + o\bigl(n^{-r}\bigr)
$$
as $n$ tends to infinitly; we 
similarly define {\em negativly approximable} with
negative real eigenvalues and $C^-_{\epsilon,\delta}$.
\end{definition}

We now state a number of conjectures, which are successively weaker.

\begin{conjecture}  For any even $d\ge 4$ we have $\cG_{\;\cdot\;,d}$ is
\begin{enumerate}
\item positively approximable to any order;
\item positively approximable to order $r(d)$, where $r(d)\ge 0$, and
$r(d)\to\infty$ as $d\to\infty$;
\item positively approximable to order $0$; and
\item positively approximable to order $0$ for $d$ sufficiently large;
\end{enumerate}
and similarly with ``postively'' replaced with ``negatively.''
\end{conjecture}
For the above conjecture, we note that $\cG_{\;\cdot\;,d}$ exhibits
$(d,\epsilon')$-tangles of order $1$ 
(see \cite{friedman_alon,friedman_kohler})
for $\epsilon'>0$ and $d\le 8$; for such reasons, we believe
that there may be a difference between small and large $d$.

Of course, the intriguing part of this conjecture is the calculation 
taking \eqref{eq:Pzero} to show
that
$$
\cP_0(u)=\sum_{k=1}^\infty P_0(k)u^{-k-1} = \frac{u^{-1}}{1-u^{-1}}
+ \frac{u^{-1}}{1-(d-1)^{-1}u^{-2}} + h(u)
$$
$$
= \frac{1}{u-1}+\frac{u}{u^2-(d-1)} + h(u),
$$
where $h(u)$ is holomorphic in $|u|>(d-1)^{-2/3}$.  
It follows that
$$
N_0 = 
\frac{1}{2\pi i}
\int_{C^+_{\epsilon,\delta}}  \cP_0(u)\,du =
1/2 ;
$$
similarly with ``positive'' replaced with ``negative.''
Hence the above conjecture would imply that the expected number of
positive, non-Ramanujan adjacency eigenvalues for a graph in
$\cG{n,d}$, with $d$ fixed and $n$ large, would tend to $1/4$.
This establishes a main point of interest.

\begin{proposition} Let $d\ge 4$ be an even integer for which
$\cG_{\;\cdot\;,d}$ is positively approximable to order $0$.  
Then,
as $n\to\infty$, the limit supremum of
the expected of positive, non-Ramanujan Hashimoto eigenvalues of
$G\in\cG_{n,d}$ is at most $1/2$; similarly, the same for non-Ramanujan
adjacency eigenvalues is at most $1/4$.
The same holds with ``positive(ly)'' replaced everywhere with
``negative(ly).''
\end{proposition}
The reason we involve the limit supremum in the above is that it 
is conceivable (although quite unlikely in our opinion) that there
is a positive expected multiplicity of the eigenvalue $(d-1)^{1/2}$
in $H_G$ for $G\in\cG_{n,d}$.

% We make the following more modest conjecture on the basis of the
% above (believing that having the $\cG_{n,d}$ expected multiplicity of
% $2(d-1)^{1/2}$ is close to zero for fixed $d$ and large $n$).
% 
% \begin{conjecture} For any fixed $d$, the $\cG_{n,d}$
% expected number of positive,
% non-Ramanujan adjacency eigenvalues tends to $1/4$ as $n\to\infty$;
% similarly with ``positive'' replaced with ``negative.''
% \end{conjecture}

We finish this secion with a few remarks considering the above conjectures.

In the usual trace methods
one estimates the expected value of $A_G^k$ or $H_G^k$ for $k$
of size proportional to $\log n$;
furthermore, the contributions to
$P_0(k)$ consist of ``single loops''
(see \cite{friedman_alon,friedman_kohler}), which cannot occur
unless $k\le n$.  Hence, the idea of fixing $n$ and formally summing
in $k$ cannot be regarded as anything but a formal summation.

We also note that for any positive integer, $m$,
$\cP_0(u)$ has poles at $(d-1)^{-(m-1)/m}\omega_m$,
where $\omega_m$ is
any $m$-th root of unity; hence this
function does not resemble $\zeta_G(u)$ for a fixed $d$-regular graph
$G$, whose poles are confined to the reals and the complex circle
$|u|=(d-1)^{-1/2}$; hence if $\cP_0(u)$ truly reflects some average
property of $\zeta_G(u)$ for $G\in\cG_{n,d}$ everwhere in $|u|>(d-1)^{-1}$,
then there is some averaging effect that makes $\cP_0(u)$ different
that the typical $\zeta_G(u)$.

\section{A Simpler Variant of the $P_i$ and $\cP_i$}
\label{se:simpler}

Part of the problem in dealing with the
$P_i$ of \eqref{eq:Pdefined} and $\cP_i$ of \eqref{eq:cPdefined}
is that the $\cP_i$ can, at least in principle (and we think likely), have
roughly $i^i$ real poles between $1$ and $(d-1)^{1/2}$, where
the $i^i$ represents roughly the number of 
{\em types} \cite{friedman_random_graphs,friedman_alon,friedman_kohler}
of order $i$ (see also \cite{puder} for similar problems).

However, the methods of \cite{friedman_kohler} 
(``mod-$S$'' functions, Section~3.5)
show that
for fixed $d$, and fixed $i$ bounded by a constant times $(d-1)^{1/3}$,
we have that $\cP_i(u)$ has poles at 
only $u=\pm 1$ and $u=\pm (d-1)^{-1/2}$ for $|u|>(d-1)^{-2/3}$.
Hence for fixed $d$ and $i$ sufficiently small, the $\cP_i(u)$ are much
simpler to analyze.
However the calculations 
\cite{friedman_random_graphs,friedman_alon,friedman_kohler} show
something a bit stronger: namely if we consider $\cP_{i,d}(u)$ and
$P_{i,d}(u)$ as depending both on $i$ and $d$, then in fact
\begin{equation}\label{eq:large_d}
P_i(k) = (d-1)^k Q_i(k,d-1) + 
\II_{{\rm even}} (d-1)^{k/2} R_i(k,d-1)  
\end{equation}
$$
+ 
\II_{{\rm odd}} (d-1)^{k/2} S_i(k,d-1)   + O(d-1)^{k/3}Ck^C
$$
for some constant $C=C(i)$, and where $Q_i(k,d-1),R_i(k,d-1),S_i(k,d-1)$ 
are polynomial
in $k$ (whose degree is bounded by a function of $i$),
whose coefficients are rational functions of $d-1$.

\begin{definition} 
The {\em large $d$ polynomials of order $i$}
are the functions $Q_i(k,d),R_i(k,d),S_i(k,d)$, determined uniquely
in \eqref{eq:large_d}.  The associated {\em approximate principle term} to
$Q_i,R_i,S_i$ is the function
$$
\widetilde P_i(k,d) = 
(d-1)^k Q_i(k,d-1) + 
\II_{k\ {\rm even}} (d-1)^{k/2} R_i(k,d-1)  +            
\II_{k\ {\rm odd}} (d-1)^{k/2} S_i(k,d-1) ,
$$
and the associated {\em approximate generating function} is
$$
\widehat\cP_i(u) = \sum_{k=0}^\infty u^{-1-k} \widehat P_i(k) (d-1)^{-k} .
$$
\end{definition}

It follows that the $\widehat\cP_i(u)$ have poles
only at $u=1$ and $u=\pm(d-1)^{-1/2}$ outside any disc about zero
of radius strictly greater than $(d-1)^{-2/3}$.

The benefit of working with the $\widehat P_i(k,d)$ and
$\widehat\cP_i(u)$ is that there are various ways of trying to
compute these functions.
For example, one can fix $k$ {\em and $n$}, and consider what
happens as $d\to\infty$.
In this case we are studying the expected number of fixed points of
strictly reduced words of length $k$ in the alphabet
$$
\Pi = \bigl\{ 
\pi_1,\pi_1^{-1},\pi_2,\pi_2^{-1},\ldots,\pi_{d/2},\pi_{d/2}^{-1}
\bigr\}
$$
with $d$ large.  If such as word has exactly one occurrence of
$\pi_j,\pi_j^{-1}$ for some $j$, which the overwhelmingly typical,
since $k$ is fixed and $d$ is large, then the expected number 
of fixed point is exactly $1$.
Hence we are lead to consider those words such that for every $j$,
either 
$\pi_i,\pi_i^{-1}$ does not occur, or it occurs at least twice.
The study of such words does not seem easy, although perhaps
this can be understood, say with the recent works
\cite{puder,puder_p}.

This type of study also seems to ressemble more closely the standard
(and much more studied)
random matrix theory than the $d$-regular spectral graph theory with
$d$ fixed; perhaps methods from random matrix theory can be applied
here.

\section{Random Graph Covering Maps and Other Models}
\label{se:covering}

We remark that \cite{friedman_alon} studies other models of random
$d$-regular graphs on $n$ vertices for $d$ and $n$ of arbitrary
partity (necessarily having half-edges if $d$ and $n$ are odd).
We remark that if $d$ is odd and $n$ is even, we can form a model of a
$d$-regular graph based on $d$ perfect matchings, called
$\cI_{n,d}$ in \cite{friedman_alon}.  A curious model of random regular
graph is $\cH_{n,d}$ (for $d$ even), which is a variant of $\cG_{n,d}$
where random permutations are chosen from the subset of permutations
whose cyclic structure consists of a single cycle of length $n$;
since such permutations cannot have self-loops (for $n>1$), for most
$d$, the probability of eigenvalues lying in the Alon region 
is larger for $\cH_{n,d}$ than for $\cG_{n,d}$.

For the rest of this section
we give a natural extension of our main conjectures
to the more general model of a random cover of a graphs.
$\cG_{n,d}$ and $\cI_{n,d}$ are special cases of a
``random degree $n$ covering map of a base graph,'' where the
base graphs are, respectively, a bouquet of $d/2$ whole-loops (requiring
$d$ to be even), and a bouquet of $d$ half-loops (where $d$ can be
either even or odd).
We shall be brief, and refer the reader to
\cite{friedman_kohler} for details.

\begin{definition} A 
\emph{morphism of directed graphs}, 
\( \varphi\from
G\rightarrow H \) is a pair \( \varphi=(\varphi_V, \varphi_E) \) for which
\( \varphi_V \from V_G \rightarrow V_H\) is a map of vertices and \(
\varphi_E \from \Edir_G \rightarrow \Edir_H \) is a map of directed edges
satisfying \( h_H(\varphi_E(e)) = \varphi_V(h_G(e)) \) and \(
t_H(\varphi_E(e)) = \varphi_V(t_G(e)) \) for all \( e \in \Edir_G \). 
We refer to the values of $\varphi_V^{-1}$ as {\em vertex fibres} of $\varphi$,
and similarly for edge fibres.
We often more simply write
\(\varphi\) instead of \(
\varphi_V\) or \( \varphi_E \).
\end{definition}

\begin{definition}
\label{de:etale}
A morphism of directed graphs \( \nu \from H \rightarrow G \)
is a \emph{{covering map}} if it is a
local isomorphism, that is
for any vertex \(w \in V_H\), the edge morphism \( \nu_E \) induces a
bijection (respectively, injection)
between \( t_H^{-1}(w) \) and \( t_G^{-1}(\nu(w)) \) and a
bijection (respectively, injection)
between \( h_H^{-1}(w) \) and \( h_G^{-1}(\nu(w)) \).
We call \( G \) the
\emph{base graph} and \( H \) a \emph{covering graph of $G$}.

If \( \nu\from H \rightarrow G \) is a covering map and \(G\) is connected,
then the \emph{degree} of \(\nu\), denoted \( [H\colon G] \), is the number
of preimages of a vertex or edge in \(G\) under \(\nu\) (which does not
depend on the vertex or edge). If \(G\) is not connected, we insist that
the number of preimages of \(\nu\) of a vertex or edge is the same,
i.e., the degree is independent of the connected component, and we will
write this number as \( [H\colon G] \).
In addition, we often refer to $H$, without $\nu$ mentioned explicitly,
as a {\em covering graph} of $G$.

A morphism of graphs is a \emph{covering map} 
if the morphism of the
underlying
directed graphs is a covering map.
\end{definition}

\begin{definition} If $\pi\from G\to B$ is a covering map of directed
graphs,
then an {\em {old function} (on $V_G$)} is a 
function on $V_G$ 
arising via pullback
from $B$, i.e., a function $f\pi$, where $f$ is a function (usually
real or complex valued), i.e., a function on $V_G$ (usually real or
complex valued) whose value depends only on the $\pi$ vertex fibres.
A {\em {new function} (on $V_G$)} is a function 
whose sum on each vertex
fibre is zero.
The space of all functions (real or complex)
on $V_G$ is a direct sum of the old and new functions, an orthogonal
direct sum on the natural inner product on $V_G$, i.e.,
$$
(f_1,f_2) = \sum_{v\in V_G} \overline{ f_1(v)} f_2(v) .
$$
The adjacency matrix,
$A_G$, viewed as an operator, takes old functions to old functions and
new functions to new functions.  The
{\em {new spectrum}} of $A_G$, which we often denote
$\specnew_B(A_G)$, is the spectrum of $A_G$
restricted to the new functions; we similarly define the 
{\em {old spectrum}}.

This discussion holds, of course, equally well if
$\pi\from G\to B$ is a covering
morphism of graphs, by doing everything over the underlying directed
graphs.
\end{definition}

We can make similar definitions for the spectrum of the Hashimoto 
eigenvalues.  
First, we observe that covering maps induce covering maps on
directed line graphs; let us state this formally (the proof is easy).

\begin{proposition} Let \( \pi \from G \to B \) be a covering map. Then \( \pi \)
induces a covering map \( \pi^{\Line} \from \Line(G) \to \Line(B) \).  
\end{proposition}
Since $\Line(G)$ and $\Line(B)$ are directed graphs, the above
discussion of new and old functions, etc., holds for
$\pi^{\Line} \from \Line(G) \to \Line(B)$; e.g., new and old functions
are functions on the vertices of $\Line(G)$, or, equivalently,
on $\Edir_G$.

\begin{definition}\label{de:new_Hashimoto}
Let \( \pi \from G \to B \) be a covering map.  We define the
{\em new Hashimoto spectrum of $G$ with respect to $B$}, denoted
$\Specnew_B(H_G)$ to be
the spectrum of the Hashimoto matrix restricted to the new
functions on $\Line(G)$, and
$\rhonew_B(H_G)$ to be the supremum of the norms of
$\Specnew_B(H_G)$.
\end{definition}
We remark that
$$
\sum_{\mu\in \Specnew_B(H_G)} \mu^k = \Tr(H_G^k)-\Tr(H_B^k),
$$
and hence the new Hashimoto spectrum is independent of the
covering map from $G$ to $B$; similarly for the new adjacency spectrum.

To any base graph, $B$, one can describe various models 
of random covering maps of degree $n$.
The simplest is to assign to each edge, $e\in E_B$, a permutation
$\pi(e)$ on ${1,\ldots,n}$ with the stipulation that
$\pi(\iota_B e)$ is the inverse permutation of $\pi(e)$;
if $e=\{e_1,e_2\}$ is not a half-loop, and $\iota e_1=e_2$, 
then we may assign
an arbitrary
random permutation to $\pi(e_1)$ and then set $\pi(e_2)$ to be
the inverse permutation of $\pi(e_1)$; if $e$ is a half-loop,
then we can assign a random perfect matching to $\pi(e)$ if $n$ is 
even, and otherwise choose $\pi(e)$ to consist of one fixed point
(which is a half-loop in the covering graph) and a random perfect
matching on the remaining $n-1$ elements.
See \cite{friedman_kohler} for a more detailed description of
this model and other ``algebraic'' models of a random covering graph
of degree $n$ over a fixed based graph.

For $d\ge 4$ even, $\cG_{n,d}$ is the above model over the base
graph which is a bouquet of $d/2$ whole-loops, and
$\cI_{n,d}$ defined at the beginning of this section
is the above model over the base graph consisting of a bouqet
of $d$ half-loops.

Again, one can make similar computations and conjectures as with
$\cG_{n,d}$.
The analogue of $P_0(k)$ for random covers of degree $n$ over
general base graph, $B$, is easily seen to be
$$
P_0(k) 
= O(kd) + \sum_{k'|k} \trace(H_B^{k'}),
$$
(see \cite{friedman_kohler}, but essentially done in
\cite{friedman_relative}, a simple variant of \cite{broder}).
As for regular graphs, the $\trace(H_B^k)$ is essentially the old
spectrum, and the new spectrum term is therefore
$$
\II_{even}(k)\trace(H_B^{k/2}) + O(k)(d-1)^{k/3} .
$$
The analogue of $\cP_0(u)$ is therefore determined by this
$$
\trace(H_B^{k/2})
$$
term for $k$ even.
There are two cases of interest:
\begin{enumerate}
\item $B$ connected and not bipartite, in which case 
$d-1$ is an eigenvalue of $H_B$, and all other eigenvalues of $H_B$
are of absolute value strictly less than $d-1$; then we get
the similar formal expansions and conjectures as before; and
\item $B$ connected and bipartite, in which case $-(d-1)$ is also
an eigenvalue, and we get a conjectured expectation of $1/2$ for the
number of positive non-Ramanujan adjacency eigenvalues.
In this case, of course, each positive eigenvalue has a corresponding
negative eigenvalue.  So, again, we would conjecture that there is
a positive probability that a random degree $n$ cover of $B$ has
two pairs (i.e., four) non-Ramanujan eigenvalues; and, again, these
two conjectures imply that for fixed $d$-regular $B$ and 
sufficiently large $n$,
a strict majority of the degree $n$ covers of $B$ are relatively Ramanujan.
\end{enumerate}

\section{Numerical Experiments}
\label{se:numerical}

Here we give some preliminary
numerical experiments done to test our conjectures for
random $4$-regular graphs.
As mentioned before, the results of \cite{miller_novikoff} indicate
that one may need graphs of many more than the $n=400,000$
vertices and fewer that we used in our experiments.

We also mention that our experiments are Bernoulli trials with
probability of success that seems to be between $.17$ and $.26$.
Hence for $n\ge 100,000$, where we test no more than $500$ random
sample graphs, the hundredths digit is not significant.
We have tested $10,000$ examples only for $n\le 10,000$.

For the model $\cG_{n,4}$, of a $d=4$-regular graph generated by
two of the $n!$ random permutations of $\{1,\ldots,n\}$, we computed
the total number of positive eigenvalues no smaller than $2(d-1)^{1/2}$.
We sampled
\begin{enumerate}
\item $10,000$ random graphs with $n=100$, $n=1000$, and $n=10,000$,
whose average number of such eigenvalues were
$1.2681$, $1.2258$, and $1.1942$, respectively;
\item $500$ random graphs with $n=100,000$, with an average of $1.176$; 
\item $250$ random graphs with $n=200,000$, with an average of $1.188$;
\item $79$ random graphs with $n=400,000$, with an average of $1.177$.
\end{enumerate}
Of course, there is always $\lambda_1(G)=d$, and there is a small
chance, roughly $1/n+O(1/n^2)$ that $G$ will be disconnected.
So our formal calculations suggest that we should see $1.25$ for 
very large $n$, or no more than this.
However there is no particularly evident convergence of these average
number of positive non-Ramanujan eigenvalues at these
values of $n$.

We also tried the model $\cH_{n,4}$, where the permutations are chosen
among the $(n-1)!$ permutations whose cyclic structure is a single
cycle of length $n$.  Since $\cH_{n,4}$ graphs are always connected,
and generally have less tangles than $\cG_{n,4}$
\cite{friedman_alon}, we felt this model
may give more representative results for the same values of $n$.
We sampled
\begin{enumerate}
\item $10,000$ random graphs with $n=100$, $n=1000$, and $n=10,000$,
whose average number of such eigenvalues were
$1.1268$, $1.161$, and $1.1693$, respectively;
\item $500$ random graphs with $n=100,000$, with an average of $1.192$; 
\item $55$ random graphs with $n=200,000$, with an average of $1.163$;
\item $87$ random graphs with $n=400,000$, with an average of $1.149$.
\end{enumerate}
Again, we see no evident convergence at this point.

% \input{se_conclusion}

% \input{se_prelim}
% \input{se_overview}
% \input{se_prince}

%    Bibliography styles amsplain or harvard are also acceptable.
\providecommand{\bysame}{\leavevmode\hbox to3em{\hrulefill}\thinspace}
\providecommand{\MR}{\relax\ifhmode\unskip\space\fi MR }
\providecommand{\MRhref}[2]{%
  \href{http://www.ams.org/mathscinet-getitem?mr=#1}{#2}
}
\providecommand{\href}[2]{#2}

\end{document}